\documentclass[reqno]{amsart} 
\usepackage{amsmath,amssymb,url,mathrsfs,dsfont,mathtools}
\usepackage{microtype,enumerate}
\usepackage[numbers,sort&compress]{natbib}

\newtheorem{theorem}{Theorem}[section]
\newtheorem{proposition}[theorem]{Proposition}
\newtheorem{corollary}[theorem]{Corollary}
\newtheorem{lemma}[theorem]{Lemma}
\newtheorem{example}[theorem]{Example}
\newtheorem{remark}[theorem]{Remark}
\newtheorem{definition}[theorem]{Definition}

\newcommand{\setone}{\mathds{1}}

\newcommand{\href}[1]{(H\ref{#1})}
\renewcommand{\phi}{\varphi}

\newcommand{\scalar}[3][]{\left(#2\mid#3\right)_{#1}}

\DeclareMathOperator{\supp}{supp}
\DeclareMathOperator{\id}{id}

\DeclareMathOperator{\diver}{div}
\newcommand{\e}{\mathrm{e}}
\DeclareMathOperator{\Capa}{Cap}
\DeclareMathOperator{\sgn}{sgn}
\DeclareMathOperator{\Artanh}{Artanh}

\numberwithin{equation}{section}

\begin{document}

\title{Isometric Lattice Homomorphisms Between Sobolev Spaces}
\author{Markus Biegert}
\address{Markus Biegert\\Institute of Applied Analysis\\University of Ulm\\89069 Ulm\\Germany}
\email{markus.biegert@uni-ulm.de}
\author{Robin Nittka}
\address{Robin Nittka\\Institute of Applied Analysis\\University of Ulm\\89069 Ulm\\Germany}
\email{robin.nittka@uni-ulm.de}
\date{August 28, 2009}


\begin{abstract}
	Given bounded domains $\Omega_1$ and $\Omega_2$ in $\mathds{R}^N$
	and an isometry $T$ from $W^{1,p}(\Omega_1)$ to $W^{1,p}(\Omega_2)$,
	we give sufficient conditions ensuring that $T$ corresponds to a rigid
	motion of the space, i.e., $Tu = \pm (u \circ \xi)$ for an isometry $\xi$,
	and that the domains are congruent. More general versions of the involved
	results are obtained along the way.
\end{abstract}

\maketitle

\section{Introduction}

Given two function spaces $F(X)$ and $F(Y)$ on carrier sets $X$ and $Y$, respectively,
it might happen that for a mapping $\xi$ from $Y$ to $X$ the composition operator
$T_\xi u \coloneqq u \circ \xi$, or more generally the \emph{weighted composition operator}
$T_{g,\xi}u \coloneqq g \cdot (u \circ \xi)$, where $g$ is a scalar-valued function,
defines a mapping from $F(X)$ to $F(Y)$.
For example, if $X$ and $Y$ are compact topological spaces, then $T_{g,\xi}$ maps $\mathrm{C}(X)$ to
$\mathrm{C}(Y)$ whenever $g$ and $\xi$ are continuous.
Often such operators $T_{g,\xi}$ have nice properties. For example, $T_{g,\xi}$ might be isometric
or a lattice homomorphism if $F(X)$ and $F(Y)$ are normed spaces or lattices.

An old question is to find properties that distinguish weighted composition operators
among all operators from $F(X)$ to $F(Y)$. For example the famous Banach-Stone
theorem (\cite[\S XI.4]{Ban87} and \cite[Theorem~83]{Sto37}) states that
for compact spaces $X$ and $Y$ each surjective linear isometry from $\mathrm{C}(X)$ to
$\mathrm{C}(Y)$ is a weighted composition operator $T_{g,\xi}$ with $\xi$ a homeomorphism from $Y$ to $X$.
Note that this says in particular that all information about $X$ as a topological space is
already encoded into the Banach space structure of $\mathrm{C}(X)$.

Another result in this direction due to Lamperti~\cite{Lam58} says that every
linear isometry of $L^p(0,1)$ into itself, $p \neq 2$, is a weighted composition operator.
Isometries on Orlicz-like spaces were considered by
John Lamperti~\cite{Lam58}, and reflexive Orlicz spaces subsequently by
G\"unter Lumer~\cite{Lum63}.

In 2006 Geoff Diestel and Alexander Koldobsky~\cite{DK06} investigated
isometries on the Sobolev space $W^{1,p}(\Omega)$, which was continued in
Goeff Diestel's subsequent article~\cite{Die}. By identifying $W^{1,p}(\Omega)$
with a subspace of an $L^p$-space and thus reducing the situation to the
setting of Lamperti's theorem, the authors were able to prove
that under additional assumptions all isometries are of the form $\pm T_\xi$
with a rigid motion $\xi$. But only few rigid motions
induce isometries in this space since the norm under consideration is
not invariant under rotations. Moreover, the reduction to Lamperti's theorem
allows to treat only the case $p \neq 2$.

Sergei Vodop'janov and Vladimir Gol'd\v{s}te\u{\i}n~\cite{VG75} showed in 1975
that an order isomorphism between the Sobolev spaces $W^{1,N}(\Omega_1)$ and
$W^{1,N}(\Omega_2)$ for domains $\Omega_1$ and $\Omega_2$ in $\mathds{R}^N$
is a composition operator if it satisfies several additional quite natural order
theoretic assumptions.

In this article we examine isometries between Sobolev spaces equipped with
a commonly used rotation invariant norm.
We can show for $p > 2$ and for $p=2$ and $N \ge 2$ that every isometry $T$ from
$W^{1,p}(\Omega_1)$ to $W^{1,p}(\Omega_2)$, where $\Omega_1$ and $\Omega_2$
are bounded open subsets of $\mathds{R}^N$,
is a weighted composition operator, whenever $W^{1,p}_0(\Omega_2) \subset TW^{1,p}_0(\Omega_1)$
and $T$ satisfies some mild additional order theoretical conditions.
The choice of the Sobolev norm does not allow
to reduce the question to the corresponding one for $L^p$-spaces
as in~\cite{DK06}, and the order assumptions are by far not strong enough
to reduce the proof to the one in~\cite{VG75}.
Instead, we use quite different techniques, based on the first author's
characterization of lattice homomorphisms between Sobolev spaces~\cite{Bie08f}.
Moreover, we can show that under slightly stronger assumptions the domains
$\Omega_1$ and $\Omega_2$ are congruent.

For most of the results we do not have to assume that $T$ is isometric,
but rather can (and do) admit larger class of operators that in some sense commute with the
$p$-Laplacian, see~\href{ass:pLaplace} in Section~\ref{sec:intertwining}.
We show in Examples~\ref{ex:p2N1} and~\ref{ex:largeN1}
that the class of operators under consideration does not only contain isometries.
We mention that a condition such as~\href{ass:pLaplace} occurs naturally
if one considers isometries of $L^2$ that commute with the Laplacian~\cite{Are02}.

The article is organized as follows.
In Section~\ref{sec:pre} we introduce some notation, which will be used freely
throughout the article, and prove some auxiliary results regarding the norm
of $W^{1,p}(\Omega)$.
The short Section~\ref{sec:iso} is devoted to the study of isometries of $W^{1,p}$-spaces.
More precisely, we show that for $p \neq 2$ every order bounded isometry is a weighted
composition operator.
In Section~\ref{sec:intertwining} we consider an operator $T$
from $W^{1,p}(\Omega_1)$ to $W^{1,p}_{\mathrm{loc}}(\Omega_2)$ such that
\[
	Tu = g \cdot (u \circ \xi)
\]
for all $u \in W^{1,p}(\Omega_1)$,
\begin{align*}
	& \int_{\Omega_2} |Tu|^{p-2} Tu \, Tv + \int_{\Omega_2} |\nabla(Tu)|^{p-2} \nabla(Tu) \nabla(Tv) \\
		& \quad = \int_{\Omega_1} |u|^{p-2} u v + \int_{\Omega_1} |\nabla u|^{p-2} \nabla u \nabla v
\end{align*}
for all $u \in W^{1,p}(\Omega_1)$ and $v \in W^{1,p}_0(\Omega_1)$ satisfying $Tv \in W^{1,p}_c(\Omega_2)$,
and
\[
	W^{1,p}_0(\Omega_2) \subset TW^{1,p}_0(\Omega_1).
\]
We show that locally $\xi$ is a rigid motion, i.e., a composition of rotation and translation, and $g$
is locally constant with $|g| \equiv 1$.
The first two of the three conditions are fulfilled if $T$ is an isometric order isomorphism
from $W^{1,p}(\Omega_1)$ to $W^{1,p}(\Omega_2)$. This is summarized in Theorem~\ref{thm:sumrep}.
In Section~\ref{sec:cong} we show that under weak additional assumptions the sets
$\Omega_1$ and $\Omega_2$ are congruent, where $\xi$ realizes this congruence.
The two main results of Section~\ref{sec:cong} are Theorems~\ref{thm:conn} and~\ref{thm:W1p0eq}.
We point out that we have the minimal possible regularity assumptions on $\Omega_1$ and $\Omega_2$
in Theorem~\ref{thm:W1p0eq}, compare Remark~\ref{rem:regcap}.
Finally, we collect some sample applications of our results in Section~\ref{sec:app}.

\section{Preliminaries and Notation}\label{sec:pre}

Let $\Omega$ be an open bounded subset of $\mathds{R}^N$,
and let $p$ be in $[1,\infty)$. We write $W^{1,p}(\Omega)$ for the Sobolev space
of all real-valued functions in $L^p(\Omega)$ such that the first order distributional derivatives are again in
$L^p(\Omega)$, and equip it with the norm
\[
	\|u\|_{W^{1,p}(\Omega)}^p \coloneqq \int_\Omega |u|^p + \int_\Omega |\nabla u|^p,
\]
where $|\cdot|$ refers to the absolute value in $\mathds{R}$ and the euclidean norm
in $\mathds{R}^N$, respectively.
We write $W^{1,p}_0(\Omega)$ for the closure of the test functions
$\mathcal{D}(\Omega) \coloneqq \mathrm{C}^\infty_c(\Omega)$ in $W^{1,p}(\Omega)$,
which is the same as the completion of $W^{1,p}_c(\Omega)$, the subspace of functions in $W^{1,p}(\Omega)$
with compact support in $\Omega$.

We use frequently the notions of the $p$-capacity $\Capa_p(E)$ of a set
$E \subset \mathds{R}^N$ and $p$-quasi continuity of a function $u$ on $\Omega$.
Definitions and the relevant results can be found for example in~\cite[\S 2.1]{MZ97}
and~\cite{Bie05,Bie09b}.
We call a set $p$-polar if its $p$-capacity is zero,
and we say that a property holds $p$-quasi everywhere
when it holds outside a $p$-polar set.

The $p$-capacity is a much finer notion than Lebesgue measure. In fact,
every $p$-polar set is a Lebesgue null set.
Moreover, it fits nicely into the theory of Sobolev spaces, so that in many respects
properties of $W^{1,p}(\Omega)$ can be formulated in
terms of the capacity. This is the reason why it appears here in a natural way.

We recall some basic facts that will be used later on without further indication.
They can be found in the aforementioned references.
Every function in $W^{1,p}(\Omega)$ admits a $p$-quasi continuous representative.
If a $p$-quasi continuous function is changed on a $p$-polar set,
it remains $p$-quasi continuous. Moreover,
two $p$-quasi continuous functions that coincide almost everywhere
do in fact coincide $p$-quasi everywhere.
Thus the $p$-quasi continuous representative is unique up to $p$-polar sets.

Note that even though formally the statements in~\cite{MZ97} are restricted to $p \le N$,
the results we use here remain true for $p > N$.
In fact, to a large extent these results become trivial in that situation
since every $p$-polar set is empty and every $p$-quasi continuous
function is continuous if $p > N$.

\smallskip

In the sequel, we need the following simple lemma. Since we could
not find it in the literature, for the sake of completeness we provide a proof here.
\begin{lemma}\label{lem:capconn}
	Let $\Omega \subset \mathds{R}^N$ be open and connected,
	and let $P$ a relatively closed subset of $\Omega$
	with $(N-1)$-dimensional Hausdorff measure zero.
	Then $\Omega \setminus P$ is connected.
\end{lemma}
\begin{proof}
	Assume $\Omega \setminus P$ to be disconnected, i.e., assume
	that there exist disjoint, non-empty, open sets $V_1$ and $V_2$ such that
	$\Omega \setminus P = V_1 \cup V_2$. Then $u \coloneqq \setone_{V_1}$
	is in $W^{1,p}(\Omega \setminus P)$ with $\nabla u = 0$.
	But since the $(N-1)$-dimensional Hausdorff measure of $P$
	is zero, $u$ is in $W^{1,p}(\Omega)$
	with $\nabla u = 0$ almost everywhere on $\Omega$~\cite[\S 1.2.5]{MP97}.
	Since $\Omega$ is connected, this implies that $u$ agrees almost everywhere
	on $\Omega$ with a constant function, which is a contradiction.
\end{proof}

The remainder of this section contains some auxiliary calculations regarding
the norm of $W^{1,p}(\Omega)$.
More precisely, we identify its first and second G\^ateaux derivative.

\begin{definition}\label{def:forms}
	Let $p\in(1,\infty)$ and $u,v,w \in W^{1,p}(\Omega)$. Define
	\[
		\mathfrak{a}_{p,\Omega}(u,v)
			\coloneqq \int_\Omega |u|^{p-2}uv + \int_\Omega |\nabla u|^{p-2}\nabla u\nabla v.
	\]
	For $p > 2$, we set
	\begin{align*}
		\mathfrak{b}_{p,\Omega}(u,v,w)
			& \coloneqq (p-1)\int_{\Omega} |u|^{p-2}vw
				+ (p-2) \int_\Omega |\nabla u|^{p-4} \scalar{\nabla u}{\nabla v} \scalar{\nabla u}{\nabla w} \\
				& \qquad + \int_\Omega |\nabla u|^{p-2} \nabla v \nabla w.
	\end{align*}
	Here, as usual, the products in the integrals are understood to equal zero where one of their factors is zero.
\end{definition}
 
\begin{remark}
	The expressions in Definition~\ref{def:forms} are also well-defined if one of the functions
	has compact support and the others lie only in $W^{1,p}_{\mathrm{loc}}(\Omega)$.
	We will use these expressions also in those cases, see for example Remark~\ref{rem:form2}.
\end{remark}

Note that $\mathfrak{a}_{p,\Omega}$ appears in the weak formulation of the differential equation
\begin{equation}\label{eq:pLaplace}
	\Delta_p w = |w|^{p-2}w,
\end{equation}
where $\Delta_pw\coloneqq \diver(|\nabla w|^{p-2}\nabla w)$ denotes
the $p$-Laplacian. More precisely, it is common to call a function $u$ in
$W^{1,p}_{\mathrm{loc}}(\Omega)$ a \emph{weak solution of~\eqref{eq:pLaplace}}
if
\begin{equation}\label{eq:pLaplaceWeak}
	\mathfrak{a}_{p,\Omega}(u,\phi) = 0
\end{equation}
holds for all $\phi$ in $\mathcal{D}(\Omega)$. If $u$ is in $W^{1,p}(\Omega)$,
then~\eqref{eq:pLaplaceWeak} holds even for all $\phi$ in $W^{1,p}_0(\Omega)$.

\begin{proposition}\label{prop:norm}
	Fix $p\in(1,\infty)$ and $u,v\in W^{1,p}(\Omega)$. Then
	\[
		\lim_{s \to 0} \frac{\|u+sv\|^p_{W^{1,p}(\Omega)}-\|u\|^p_{W^{1,p}(\Omega)}}{s}
			= p \, \mathfrak{a}_{p,\Omega}(u,v)
	\]
\end{proposition}

\begin{proof}
	For $x$ and $y$ in $\mathds{R}^N$ define
	\[
		g(s;x,y) \coloneqq |x+sy|^p.
	\]
	It is easy to check that $g$ is continuously differentiable with respect to $s$
	and
	\[
		g'(s;x,y) = p \left|x+sy\right|^{p-2} \scalar{x+sy}{y}.
	\]
	Hence
	\begin{align*}
		\lim_{s \to 0} \frac{|\nabla u+s\nabla v|^p - |\nabla u|^p}{s}
			& = \lim_{s \to 0} \frac{g(s;\nabla u,\nabla v)-g(0;\nabla u,\nabla v)}{s-0} \\
			& = g'(0;\nabla u,\nabla v)
			= p \left|\nabla u\right|^{p-2} \nabla u \cdot \nabla v,
	\end{align*}
	where the limit is understood pointwise.
	For $s \in (-1,1)$, $s \neq 0$, it follows from the mean value theorem that
	\begin{align*}
		\biggl| \frac{|\nabla u+s\nabla v|^p - |\nabla u|^p}{s} \biggr|
			& = \bigl| g'(\xi;\nabla u,\nabla v) \bigr|
			\le p \left| \nabla u + \xi \nabla v \right|^{p-1} \left| \nabla v \right| \\
			& \le p \bigl( |\nabla u| + |\nabla v| \bigr)^{p-1} |\nabla v|,
	\end{align*}
	where $\xi \coloneqq \xi(s,\nabla u,\nabla v)$ takes values in $(-1,1)$.
	Since the right hand side is in $L^1(\Omega)$ due to H\"older's inequality,
	it follows from Lebesgue's dominated convergence theorem that
	\begin{equation}\label{prop:norm:nablau}
		\lim_{s \to 0} \int_\Omega \frac{|\nabla u+s\nabla v|^p - |\nabla u|^p}{s}
			= p \int_\Omega \left|\nabla u\right|^{p-2} \nabla u \nabla v.
	\end{equation}
	Replacing in the above argument $\nabla u$ by $u$ and $\nabla v$ by $v$,
	we also obtain
	\begin{equation}\label{prop:norm:u}
		\lim_{s \to 0} \int_\Omega \frac{|u+sv|^p - |u|^p}{s} = p \int_\Omega |u|^{p-2} uv.
	\end{equation}
	Now the claim follows by adding~\eqref{prop:norm:nablau} and~\eqref{prop:norm:u}.
\end{proof}

\begin{proposition}\label{prop:form}
	Fix $p\in(2,\infty)$ and $u,v,w\in W^{1,p}(\Omega)$. Then
	\[
		\lim_{s \to 0} \frac{\mathfrak{a}_{p,\Omega}(u+sv,w)-\mathfrak{a}_{p,\Omega}(u,w)}{s}
			\to \mathfrak{b}_{p,\Omega}(u,v,w)
	\]
\end{proposition}

\begin{proof}
	For $x$ and $y$ in $\mathds{R}^N$ define
	\[
		g(s;x,y) \coloneqq |x+sy|^{p-2} (x+sy).
	\]
	It is easy to check that $g$ is continuously differentiable with respect to $s$
	and
	\[
		g'(s;x,y) = (p-2) \left|x+sy\right|^{p-4} \scalar{x+sy}{y} (x+sy) + \left|x+sy\right|^{p-2} y.
	\]
	Hence
	\begin{align*}
		& \lim_{s \to 0} \frac{|\nabla u+s\nabla v|^{p-2}(\nabla u + s\nabla v) - |\nabla u|^{p-2} \nabla u}{s} \\
			& \qquad = \lim_{s \to 0} \frac{g(s;\nabla u,\nabla v)-g(0;\nabla u,\nabla v)}{s-0}
			= g'(0;\nabla u,\nabla v) \\
			& \qquad = (p-2) \left|\nabla u\right|^{p-4} \scalar{\nabla u}{\nabla v} \nabla u + \left|\nabla u\right|^{p-2} \nabla v.
	\end{align*}
	where the limit is understood pointwise.
	For $s \in (-1,1)$, $s \neq 0$, it follows from the mean value theorem that
	\begin{align*}
		& \frac{|\nabla u+s\nabla v|^{p-2}(\nabla u + s\nabla v) - |\nabla u|^{p-2} \nabla u}{s}
			= \bigl| g'(\xi;\nabla u,\nabla v) \bigr| \\
			& \qquad \le (p-1) \bigl| \nabla u + \xi \nabla v \bigr|^{p-2} \left|\nabla v\right|
			\le (p-1) \bigl( |\nabla u| + |\nabla v| \bigr)^{p-2} \left|\nabla v\right|
	\end{align*}
	where $\xi \coloneqq \xi(s,\nabla u,\nabla v)$ takes values in $(-1,1)$.
	Since the right hand side is in $L^{p'}(\Omega)$ due to H\"older's inequality,
	it follows from Lebesgue's dominated convergence theorem that
	\begin{equation}\label{prop:form:nablau}
		\begin{aligned}
			& \lim_{s \to 0} \int_\Omega \frac{|\nabla u+s\nabla v|^{p-2}(\nabla u + s\nabla v) - |\nabla u|^{p-2} \nabla u}{s} \nabla w \\
				& \qquad = (p-2) \int_\Omega \left|\nabla u\right|^{p-4} \scalar{\nabla u}{\nabla v} \nabla u \nabla w
					+ \int_\Omega \left|\nabla u\right|^{p-2} \nabla v \nabla w.
		\end{aligned}
	\end{equation}
	Replacing in the above argument $\nabla u$ by $u$, $\nabla v$ by $v$, and $\nabla w$ by $w$, and taking
	into account that $\scalar{u}{v}u = |u|^2 v$, we also obtain
	\begin{equation}\label{prop:form:u}
		\lim_{s \to 0} \int_\Omega \frac{|u+sv|^{p-2}(u+sv) - |u|^{p-2}u}{s} w
			= (p-1) \int_\Omega |u|^{p-2} vw.
	\end{equation}
	Now the claim follows by adding~\eqref{prop:form:nablau} and~\eqref{prop:form:u}.
\end{proof}

\begin{remark}\label{rem:form2}
	Proposition~\ref{prop:form} remains true for $u, v \in W^{1,p}_{\mathrm{loc}}(\Omega)$
	if $w$ has compact support in $\Omega$, with exactly the same proof.
\end{remark}

\begin{corollary}\label{cor:intertwine}
	Let $p \in [1,\infty)$, and let $T$ be an isometry from $W^{1,p}(\Omega_1)$ to $W^{1,p}(\Omega_2)$.
	Then
	\[
		\mathfrak{a}_{p,\Omega_1}(u,v) = \mathfrak{a}_{p,\Omega_2}(Tu,Tv) \text{ for all } u,v \in W^{1,p}(\Omega_1).
	\]
	In particular, if $u \in W^{1,p}(\Omega_1)$ is a weak solution of~\eqref{eq:pLaplace} on
	$\Omega_1$, then $Tu$ is a weak solution of~\eqref{eq:pLaplace} on $\Omega_2$
	provided that $\mathcal{D}(\Omega_2) \subset TW^{1,p}_0(\Omega_1)$.
\end{corollary}

\section{Isometries between Sobolev Spaces}\label{sec:iso}

We will need the following vector-valued version of Clarkson's famous inequality,
which is an easy consequence of~\cite[Corollary~2.5]{TK97} as explained in
the introduction of~\cite{Cho01}.

\begin{lemma}\label{lem:lamperti}
	Let $f,g \in L^p(\Omega; \mathds{R}^N)$. Then
	\[
		\|f+g\|_p^p + \|f-g\|_p^p \ge 2\|f\|_p^p + 2\|g\|_p^p,
	\]
	if $p \ge 2$, whereas for $p \le 2$ the opposite inequality holds.
\end{lemma}

From this estimate we can deduce that isometries of Sobolev spaces
respect the disjointness of sets.

\begin{theorem}\label{thm:disj}
	Let $\Omega_1,\Omega_2\subset\mathds{R}^N$ be open sets and let $p\in[1,\infty)$, $p\neq2$.
	If $T$ is a linear isometry from $W^{1,p}(\Omega_1)$ to $W^{1,p}(\Omega_2)$, then
	$T$ is disjointness preserving, i.e., for $u,v\in W^{1,p}(\Omega_2)$
	such that $|u| \wedge |v|=0$ we have $|Tu| \wedge |Tv|=0$.
\end{theorem}

\begin{proof}
	Assume $p > 2$ and let
	$u$ and $v$ be in $W^{1,p}(\Omega_1)$ with $|u|\wedge|v|=0$.
	Then
	\begin{align*}
		 \|Tu\|_{W^{1,p}(\Omega_2)}^p + \|Tv\|_{W^{1,p}(\Omega_2)}^p 
			 &= \|u\|_{W^{1,p}(\Omega_1)}^p + \|v\|_{W^{1,p}(\Omega_1)}^p \\
			 &= \|u+v\|^p_{W^{1,p}(\Omega_1)} = \|Tu+Tv\|^p_{W^{1,p}(\Omega_2)}.
	\end{align*}
	Replacing $v$ by $-v$ and adding up the resulting two equalities we obtain
	\[
			\|Tu+Tv\|^p_{W^{1,p}(\Omega_2)}+\|Tu-Tv\|^p_{W^{1,p}(\Omega_2)}
				= 2\|Tu\|_{W^{1,p}(\Omega_2)}^p + 2\|Tv\|_{W^{1,p}(\Omega_2)}^p,
	\]
	Applying Lemma~\ref{lem:lamperti} to $\nabla(Tu)$ and $\nabla(Tv)$,
	by the definition of the norm this yields
	\begin{align*}
		\|Tu+Tv\|_{L^p(\Omega_2)}^p + \|Tu-Tv\|_{L^p(\Omega_2)}^p
			\le 2 \|Tu\|_{L^p(\Omega_2)}^p + 2\|Tv\|_{L^p(\Omega_2)}^p.
	\end{align*}
	By Clarkson's inequality~\cite[Corollary~2.1]{Lam58}
	this implies $Tu \cdot Tv = 0$ almost everywhere, which is an equivalent formulation
	of disjointness.

	The proof for the case $p < 2$ is similar.
\end{proof}

\begin{remark}
	Due to strict convexity, for $p > 1$
	the linearity condition in the previous theorem is (up to translation) automatically
	fulfilled, see~\cite{Vae03}.
\end{remark}

As a consequence of the last theorem we obtain a variant of Lamperti's theorem,
Theorem~\ref{thm:HomCoMu}.
Using one of the other representation theorems in~\cite{Bie08f}, we could also deduce
related, but slightly different results.

For this theorem we have to assume the operator to be order bounded.
This is automatic in Lamperti's original setting of $L^p$-spaces by results of
Arendt~\cite[Theorem~2.5]{Are83} and Abramovich~\cite[\S 2]{Abr83},
see also~\cite[Remark~5]{deP84}. But it is not clear
whether this assumption is redundant also for Sobolev spaces.

\begin{theorem}\label{thm:HomCoMu}
	Let $\Omega_1,\Omega_2\subset\mathds{R}^N$ be bounded open sets and let $p\in[1,\infty)$, $p\neq2$.
	Let $T$ be a linear isometry from $W^{1,p}(\Omega_1)$ to $W^{1,p}(\Omega_2)$ that is
	order bounded from $W^{1,p}(\Omega_1)$ to $L^p(\Omega_2)$.
	Then $T$ is a weighted composition operator, i.e., there exists a function $\xi$ from
	$\Omega_1$ to $\Omega_2$ and a scalar-valued function $g$ on $\Omega_2$ such that
	$Tu = g \cdot (u \circ \xi)$ almost everywhere for each $u \in W^{1,p}(\Omega_1)$.
	In particular, the preimage under $\xi$ of a Lebesgue null set is a Lebesgue null set.
\end{theorem}

\begin{proof}
	By Theorem~\ref{thm:disj}, $T$ preserves disjointness
	as an operator from $W^{1,p}(\Omega_1)$ to $W^{1,p}(\Omega_2)$ and thus also as an operator
	from $W^{1,p}(\Omega_1)$ to $L^p(\Omega_2)$.
	Hence, by a result due to Meyer~\cite[Corollary~4]{deP84},
	$T = T^+ - T^-$ for lattice homomorphisms $T^+$ and $T^-$
	from $W^{1,p}(\Omega_1)$ to $L^p(\Omega_2)$ that
	satisfy $T^+u = (Tu)^+$ and $T^-u = (Tu)^-$ for $u \ge 0$.
	These operators are weighted composition operators~\cite[Theorem~4.13]{Bie08f}, i.e.,
	$T^+u = g_1 \cdot (u \circ \xi_1)$ and $T^-u = g_2 \cdot (u \circ \xi_2)$,
	where $g_1 \ge 0$ and $g_2 \ge 0$.
	Now
	\[
		g_1 = T^+\setone_{\Omega_1} = (T\setone_{\Omega_1})^+ = (g_1 - g_2)^+,
	\]
	so $g_2 = 0$ almost everywhere on $\{g_1 > 0\}$. Similarly,
	$g_1 = 0$ almost everywhere on $\{g_2 > 0\}$. Hence for
	\[
		\xi \coloneqq \begin{cases} \xi_1, & g_1 > 0, \\ \xi_2, & \text{otherwise} \end{cases}
	\]
	and $g \coloneqq g_1 - g_2$ we have $Tu = g \cdot (u \circ \xi)$ almost
	everywhere.
\end{proof}

\section{Intertwining Weighted Composition Operators}\label{sec:intertwining}

Let $\Omega_1,\Omega_2\subset\mathds{R}^N$ be bounded open sets, and let $1 < p < \infty$.
In this section, $T$ will always denote an operator from $W^{1,p}(\Omega_1)$ to $W^{1,p}_{\mathrm{loc}}(\Omega_2)$.
Consider the following three assumptions:
\begin{enumerate}[(H1)]
\item\label{ass:CoMu}
	There exist $\mathrm{C}^1$-functions
	$\xi\colon \Omega_2 \to \mathds{R}^N$ and $g\colon \Omega_2 \to \mathds{R}$
	such that $\xi(y) \in \Omega_1$ for almost all $y \in \Omega_2$,
	\[
		\{ y \in \Omega_2 : g(y) \neq 0 \text{ and } \xi(y) \in N \}
	\]
	is a Lebesgue null set for every Lebesgue null set $N \subset \Omega_1$,
	and $Tu = g \cdot (u \circ \xi)$ almost everywhere for each $u \in W^{1,p}(\Omega_1)$;
\item\label{ass:pLaplace}
	$\mathfrak{a}_{p,\Omega_2}(Tu,Tv)=\mathfrak{a}_{p,\Omega_1}(u,v)$
		for all $u \in W^{1,p}(\Omega_1)$ and $v \in W^{1,p}_0(\Omega_1)$
		satisfying $Tv \in W^{1,p}_c(\Omega_2)$.
\item\label{ass:W1p0}
	$W^{1,p}_0(\Omega_2) \subset TW^{1,p}_0(\Omega_1)$;
\end{enumerate}

Our aim is to show that if these three assumptions are satisfied, then
locally $\xi$ is a rigid motion, i.e., a composition of a
translation and a rotation, and $g$ is locally constant with $|g| = 1$.

\smallskip

Maybe the assumptions on $T$ seem strange at first glance.
Therefore a few remarks seem to be in order.
\begin{remark}\label{rem:norm}\mbox{}
\begin{enumerate}[(a)]
\item
	If follows from~\href{ass:CoMu} that $T$ is linear.
\item
	In~\href{ass:CoMu} the two conditions on the behavior of $\xi$ are needed
	to ensure that the function $u \circ \xi$ is well-defined almost
	everywhere.
\item
	We will show in Proposition~\ref{prop:CoMu} that if $T$ is a weighted composition
	operator, i.e., $Tu = g \cdot (u \circ \xi)$ for some (not necessarily smooth)
	functions $g$ and $\xi$, and if~\href{ass:pLaplace} and~\href{ass:W1p0}
	are satisfied, then $g$ and $\xi$ agree almost everywhere with
	continuously differentiable functions. So under the other two assumptions, \href{ass:CoMu}
	is equivalent to $T$ being a weighted composition operator.
\item
	If $p \neq 2$ and $T$ is an isometric, order bounded, linear mapping from
	$W^{1,p}(\Omega_1)$ to $W^{1,p}(\Omega_2)$,
	then by Theorem~\ref{thm:HomCoMu} and Corollary~\ref{cor:intertwine}
	the assumptions~\href{ass:CoMu} and~\href{ass:pLaplace} are fulfilled.
\item
	Assumption~\href{ass:W1p0} ensures
	a sufficiently large collection of functions in the image of $T$.
	Moreover, \href{ass:W1p0} is needed in addition to~\href{ass:pLaplace}
	in order to deduce $T$ maps solutions of~\eqref{eq:pLaplace} on $\Omega_1$
	to solutions of~\eqref{eq:pLaplace} on $\Omega_2$
	and is from this point of view rather natural.
\end{enumerate}
\end{remark}

\begin{lemma}\label{lem:cont}
	If $T$ is a weighted composition operator from $W^{1,p}(\Omega_1)$
	to $W^{1,p}_{\mathrm{loc}}(\Omega_2)$, then $T$ is continuous.
\end{lemma}
\begin{proof}
	By assumption $Tu = g \cdot (u \circ \xi)$ for all $u \in W^{1,p}(\Omega_1)$,
	where $g$ and $\xi$ are functions on $\Omega_2$.
	Then $T$ is the difference
	of two positive operators $T = T^+ - T^-$, where $T^+u \coloneqq g^+ \cdot (u \circ \xi)$
	and $T^-u \coloneqq g^- \cdot (u \circ \xi)$. Note that $T^+u = (T u^+)^+ - (T u^-)^+$.
	Hence $T^+$ and $T^-$ are positive linear mappings into $W^{1,p}_{\mathrm{loc}}(\Omega_2)$.
	Thus the operators $T^+$ and $T^-$ are continuous~\cite[Theorem~2.32]{AT07}.
	Hence $T$ is continuous.
\end{proof}

\begin{lemma}\label{lem:harm}
	Let $p\in (1,\infty)$, $\alpha = (\alpha_i) \in \mathds{R}^N$ such that
	$|\alpha|^p = (p-1)^{-1}$, and define
	$u \coloneqq \e^f$, $f(x) \coloneqq \sum_{i=1}^N \alpha_i x_i$.
	Then $u \in \mathrm{C}^\infty(\mathds{R}^N)$ is a classical solution of~\eqref{eq:pLaplace}.
\end{lemma}

\begin{proof}
	Since $D_i u = \alpha_i u$,
	\[
		|\nabla u|^{p-2} D_i u = \alpha_i |\alpha|^{p-2} u^{p-1},
	\]
	thus
	\begin{align*}
		 \Delta_p u
		 	= \sum_{i=1}^N D_i(|\nabla u|^{p-2} D_i u)
			= \sum_{i=1}^N \alpha_i |\alpha|^{p-2} (p-1) u^{p-2} \alpha_i u
			= |\alpha|^p (p-1) u^{p-1}.
	\end{align*}
	Hence $\Delta_p u = |u|^{p-2} u$ according to the choice of $\alpha$ and
	since $u \ge 0$.
\end{proof}

\begin{proposition}\label{prop:CoMu}
	Assume that $T$ is a weighted composition operator from
	$W^{1,p}(\Omega_1)$ to $W^{1,p}_{\mathrm{loc}}(\Omega_2)$
	that satisfies~\href{ass:pLaplace} and~\href{ass:W1p0}.
	Then $Tu$ satisfies~\href{ass:CoMu}, and the function $g$
	in~\href{ass:CoMu} has no zero in $\Omega_2$.
\end{proposition}

\begin{proof}
	By assumption there exist functions $g$ and $\xi$ on $\Omega_2$
	such that $Tu = g \cdot (u \circ \xi)$ almost everywhere
	for all $u \in W^{1,p}(\Omega_1)$.
	We have to show that $g$ and $\xi$ agree almost everywhere
	with continuously differentiable functions, say $h$ and $\zeta$,
	and that $h$ has no zero in $\Omega_2$.

	Let $j\in\{1,\dots,N\}$ be fixed, $\alpha\coloneqq (p-1)^{-1/p}$,
	$u_{\pm,j}(x) \coloneqq \e^{\pm\alpha x_j}$.
	By~\href{ass:W1p0}, for each $\phi$ in $W^{1,p}_c(\Omega_2)$ there exists
	$w \in W^{1,p}_0(\Omega_1)$ such that $Tw = \phi$.
	By Lemma~\ref{lem:harm} the function $u_{\pm,j}$ is a solution of~\eqref{eq:pLaplace},
	so we obtain from~\href{ass:pLaplace} and the divergence theorem that
	\[
		\mathfrak{a}_{p,\Omega_2}(Tu_{\pm,j},\phi) = \mathfrak{a}_{p,\Omega_1}(u_{\pm,j},w) = 0.
	\]
	Since $\phi$ was arbitrary,
	\begin{equation}\label{eq:vformula}
		v_{\pm,j}
			\coloneqq Tu_{\pm}
			= g \cdot (u_{\pm} \circ \xi)
			= g \, \e^{\pm \alpha x_j}
	\end{equation}
	is a weak solution of~\eqref{eq:pLaplace} on $\Omega_2$.
	Elliptic local regularity theory~\cite{Lew83} implies that $v_{\pm,j}$ admits a continuously differentiable
	representative, which we again denote by $v_{\pm,j}$.

	Note that by~\eqref{eq:vformula}
	$g = \sgn(v_{+,j}) ( v_{+,j} \cdot v_{-,j} )^{1/2}$ almost everywhere for every $j$
	and that the right hand side is a continuous function for each $j$.
	Hence the right hand side does in fact not depend on $j$ so that
	\begin{equation}\label{eq:defcontg}
		h \coloneqq \sgn(v_{+,j}) ( v_{+,j} \cdot v_{-,j} )^{1/2}
	\end{equation}
	is well-defined and represents a continuous function that coincides
	with $g$ almost everywhere.
	
	Let $P \coloneqq \{ h = 0 \}$. Then $P$ is a relatively closed subset of $\Omega_2$,
	and by~\eqref{eq:defcontg} we have that $P = \{ v_{\pm,j} = 0 \}$ and that $h$
	is continuously differentiable on $\Omega_2 \setminus P$.
	Also the function
	\[
		\zeta_j \coloneqq \frac{1}{2\alpha} \log\Bigl( \frac{v_+}{v_-} \Bigr)
	\]
	is well-defined and continuously differentiable on $\Omega_2 \setminus P$,
	and by~\eqref{eq:vformula} the vector-valued function $\zeta$ coincides with $\xi$
	almost everywhere on $\Omega_2 \setminus P$. We may define $\zeta$ arbitrarily
	on $P$.

	We have shown that $T\phi = h \cdot (\phi \circ \zeta)$ almost everywhere for
	each $\phi \in W^{1,p}(\Omega_1)$. Using~\href{ass:W1p0}, we see from this that
	every test function $\psi \in \mathcal{D}(\Omega_2)$ is zero
	almost everywhere on $P$. Hence $P$ is a Lebesgue null set.
	Moreover, the function $h \cdot (\phi \circ \zeta)$ is
	continuous at each point of $P$ if $\phi$ is bounded, and it is continuous
	at each point of $\Omega_2 \setminus P$ if $\phi$ is continuous. Hence
	$T\phi$ has a continuous representative whenever $\phi$ is in
	$W^{1,p}(\Omega_1) \cap \mathrm{C}(\Omega_1) \cap L^\infty(\Omega_1)$,
	and this representative vanishes on $P$.

	We are going to show that $P$ is a $p$-polar set.
	Let $\psi$ be a $p$-quasi continuous representative of an arbitrary function
	in $W^{1,p}_0(\Omega_2)$.
	By~\href{ass:W1p0} there exists $\phi$ in $W^{1,p}_0(\Omega_1)$ such that $T\phi = \psi$.
	Let $(\phi_n)$ be a sequence of test functions that converges to $\phi$ in $W^{1,p}(\Omega_1)$,
	and let $\psi_n$ be the continuous representative of $T\phi_n$, which by
	the previous paragraph exists and vanishes on $P$.
	By continuity of $T$ the sequence $(\psi_n)$ converges to $\psi$
	in $W^{1,p}_{\mathrm{loc}}(\Omega_2)$, see Lemma~\ref{lem:cont}.
	Passing to a subsequence we can arrange that $(\psi_n)$ converges
	to $\psi$ $p$-quasi everywhere~\cite[Lemma~2.19]{MZ97}, showing that
	$\psi = 0$ $p$-quasi everywhere on $P$.
	Thus the restriction of $\psi$ to $\Omega_2 \setminus P$ is in
	$W^{1,p}_0(U_2 \setminus P)$~\cite[Lemma~2.26]{MZ97}.
	Since this is true for any $\psi$, $P$ is $p$-polar~\cite[Theorem~2.15]{MZ97}.

	Now we show that in fact $P = \emptyset$, thus finishing the proof.
	For this let $V$ be an open, connected subset
	of $\Omega_2$. Then also $V \setminus P$ is connected, see Lemma~\ref{lem:capconn}.
	Since $h$ is continuous and has no zeros in $V \setminus P$, $h$ does
	not change sign on this set. So assume that $h \ge 0$ on $V$,
	the case $h \le 0$ being analogous.
	Then also $v_{\pm,j} \ge 0$ on $V$ by~\eqref{eq:defcontg}.
	But $v_{\pm,j}$ is a weak solution of~\eqref{eq:pLaplace}, so
	either $v=0$ on $V$ or $v$ is strictly positive on $V$~\cite[Theorem~1.1.1]{PS07}.
	The case $v=0$ on $V$ is impossible since then $V \subset P$,
	which is inconsistent with $P$ being a Lebesgue null set.
	So in fact $v$ is strictly positive on $V$, showing that
	$P \cap V = \emptyset$.
	Since this argument is true for any connected open subset $V$
	of $\Omega_2$, we obtain that $P = \emptyset$.
\end{proof}

We would like to show that $g$ is locally constant
and $\xi$ is locally a rigid motion. Unfortunately, our proof does
not cover all possible cases. However, we do have proofs for the case $p > 2$
and the case $p = 2$ and $N \ge 2$, and these two situations are rather different.
Moreover, we show that for $p=2$ and $N=1$ the conclusion is false, see
Example~\ref{ex:p2N1}.

\smallskip

We start with the case $p=2$.
\begin{lemma}\label{lem:p2prop}
	Let $p=2$, and let $T$ satisfy~\href{ass:CoMu}, \href{ass:pLaplace}, and~\href{ass:W1p0}.
	Then $\nabla \xi_i \perp \nabla \xi_j$ and
	$|\nabla \xi_i| = |\nabla \xi_j|$ for all $i \neq j$.
	Writing $c \coloneqq |\nabla \xi_i|$, we have
	$g - \Delta g = g c^2 = c^N / g$ and
	$g \, \Delta \xi_j + 2 \nabla g \, \nabla \xi_j = 0$.
\end{lemma}
\begin{proof}
	First note that linear elliptic regularity theory and the proof
	of Proposition~\ref{prop:CoMu} yield that $g$ and $\xi$ are $\mathrm{C}^\infty$-functions.
	
	Let $u = \e^f$ be as in Lemma~\ref{lem:harm} for $p=2$, so that
	$\Delta u = u$. Fix an arbitrary test function $\eta \in \mathcal{D}(\Omega_2)$.
	Then also $\e^{-f \circ \xi} \eta$ is in $W^{1,p}_c(\Omega_2)$. Hence
	by~\href{ass:W1p0} there exists $\phi$ in $W^{1,p}_0(\Omega_1)$ such that $T\phi = \e^{-f \circ \xi}\eta$,
	i.e., $\eta = (u \circ \xi) T\phi$. Thus
	\[
		\nabla\eta = (u \circ \xi) \nabla(f \circ \xi) T\phi + (u \circ \xi) \nabla(T\phi)
			= \nabla(f \circ \xi) \eta + (u \circ \xi) \nabla(T\phi).
	\]
	Since $u$ is a weak solution of~\eqref{eq:pLaplace}, we obtain from~\href{ass:pLaplace} that
	\begin{align*}
		0 & = \mathfrak{a}_{2,\Omega_1}(u,\phi)
			= \mathfrak{a}_{2,\Omega_2}(Tu,T\phi) \\
			& = \int_{\Omega_2} g (u \circ \xi) T\phi + \int_{\Omega_2} \nabla g \, (u \circ \xi) \, \nabla(T\phi) + \int_{\Omega_2} g \, (u \circ \xi) \, \nabla(f \circ \xi) \, \nabla(T\phi) \\
			& = \int_{\Omega_2} g \eta + \int_{\Omega_2} \nabla g \, \nabla\eta - \int_{\Omega_2} \nabla g \, \nabla(f \circ \xi) \, \eta
				+ \int_{\Omega_2} g \, \nabla(f \circ \xi) \nabla\eta - \int_{\Omega_2} g |\nabla(f \circ \xi)|^2 \eta.
	\end{align*}
	By the same argument the last identity holds also with $f$ replaced by $-f$.
	Adding these two equations and dividing by $2$ yields
	\begin{equation}\label{eq:p2first}
		\int_{\Omega_2} g \eta + \int_{\Omega_2} \nabla g \, \nabla \eta = \int_{\Omega_2} g |\nabla(f \circ \xi)|^2 \, \eta.
	\end{equation}
	whereas subtracting one from the other and dividing by $2$ gives
	\begin{equation}\label{eq:p2second}
		\int_{\Omega_2} \nabla g \, \nabla(f \circ \xi) \, \eta = \int_{\Omega_2} g \, \nabla(f \circ \xi) \, \nabla \eta
			= -\int_{\Omega_2} g \, \Delta (f \circ \xi) \eta - \int_{\Omega_2} \nabla g \, \nabla (f \circ \xi) \, \eta,
	\end{equation}
	where we used the divergence theorem.
	
	Since~\eqref{eq:p2first} is true for every test function $\eta$, we have in fact the pointwise identity
	$g - \Delta g = g \, |\nabla(f \circ \xi)|^2$, so that $c \coloneqq c(y) \coloneqq |\nabla(f \circ \xi)|$
	does not depend on the choice of $f$ in Lemma~\ref{lem:harm}.
	For $f(x) = x_i$ we see that $c = |\nabla \xi_i|$.
	Considering $f(x) \coloneqq \alpha x_i + \alpha x_j$ for $i \neq j$, where $2\alpha^2 = 1$,
	we see that
	\[
		c^2 = |\alpha \nabla \xi_i + \alpha \nabla \xi_j|^2
			= \alpha^2 \bigl( |\nabla \xi_i|^2 + 2 \nabla \xi_i \, \nabla \xi_j + |\nabla \xi_j|^2 \bigr)
			= c^2 + \nabla \xi_i \, \nabla \xi_j,
	\]
	which shows that $\nabla\xi_i$ and $\nabla\xi_j$ are orthogonal for $i \neq j$.
	Moreover, since \eqref{eq:p2second} is true for every test function,
	$g \, \Delta\xi_j + 2 \nabla g \, \nabla\xi_j = 0$ on $\Omega_2$,
	where we have set $f(x) \coloneqq x_j$.

	Fix $y_0 \in \Omega_2$ such that $c(y_0) \neq 0$.
	Then $\xi$ is invertible on a neighborhood $U_2$ of $y_0$.
	Let $\eta$ be a test function on $U_2$. By~\href{ass:W1p0} there
	exists $\phi \in W^{1,p}_0(\Omega_1)$ such that $T\phi = \eta$.
	Replacing $\phi$ by $\phi \, \setone_{\xi(U_2)} \in W^{1,p}_0(\Omega_1)$,
	we can assume that $\phi = 0$ outside $\xi(U_2)$.
	Then
	\begin{align*}
		\int_{U_2} \frac{c^N}{g} \eta
			& = \int_{U_2} (\phi \circ \xi) |\det \xi'|
			= \int_{\Omega_1} \phi
			= \mathfrak{a}_{2,\Omega_1}(\setone, \phi) \\
			& = \mathfrak{a}_{2,\Omega_2}(g, \eta)
			= \int_{\Omega_2} g \eta + \int_{\Omega_2} \nabla g \, \nabla\eta
			= \int_{\Omega_2} g \eta - \int_{\Omega_2} \Delta g \, \eta.
	\end{align*}
	Since $\eta$ was an arbitrary test function,
	$g - \Delta g = c^N / g$ on $U_2$, and in particular at $y_0$.
	Since the preimage of null sets under $\xi$ are null sets, $\xi$ is not constant
	on an open set. Hence $\{ c \neq 0 \}$ is dense in $\Omega_2$.
	Thus $g - \Delta g = c^N / g$ on the whole set $\Omega_2$ by continuity.
\end{proof}

\begin{proposition}
	Let $p=2$, $N=2$, and let $T$ satisfy~\href{ass:CoMu}, \href{ass:pLaplace}, and~\href{ass:W1p0}.
	Then $\nabla g = 0$ on $\Omega_2$.
\end{proposition}
\begin{proof}
	Since $g c^2 = c^2 / g$ by Lemma~\ref{lem:p2prop},
	$g$ takes values in $\{-1,1\}$. By continuity of $g$ this shows that $g$ is locally constant,
	i.e., $\nabla g = 0$.
\end{proof}

\begin{proposition}
	Let $p=2$, $N \ge 3$, and let $T$ satisfy~\href{ass:CoMu}, \href{ass:pLaplace}, and~\href{ass:W1p0}.
	Then $\nabla g = 0$ on $\Omega_2$.
\end{proposition}
\begin{proof}
	By Lemma~\ref{lem:p2prop} the function $\xi$ is a conformal
	mapping and hence a M\"obius transformation by Liouville's theorem~\cite[\S 3.8]{Bea83}.
	Thus either $\xi$ is affine linear or $\xi = \tau_1 \circ \sigma \circ \tau_2$,
	where $\tau_i(x) = k_i A_ix + y_i$ for an orthogonal matrix $A_i$, a scalar $k_i > 0$,
	and a vector $y_i \in \mathds{R}^N$, and where $\sigma$ is the inversion at the
	unit sphere, i.e., $\sigma(x) = \frac{x}{|x|^2}$,
	cf.~\cite[Theorem~3.5.1 and Equation~(3.1.4)]{Bea83}.
	
	First consider the case that $\xi$ is affine linear. Then $\Delta \xi_j = 0$
	for all $j$ and thus $\nabla g$ is orthogonal to each $\nabla \xi_j$
	by Lemma~\ref{lem:p2prop}. Since the $\nabla \xi_j$ are pairwise
	orthogonal, $\nabla g = 0$ on $\{ |\nabla \xi_j| \neq 0 \}$, which implies
	$\nabla g = 0$ on $\Omega_2$ by continuity since $\{ |\nabla \xi_j| \neq 0 \}$
	is dense in $\Omega_2$.

	Now assume that $\xi = \tau_1 \circ \sigma \circ \tau_2$.
	We will show that this assumption is contradictory, thus proving the claim.
	By the chain rule,
	\[
		c \coloneqq |\nabla \xi_j| = |\xi'|
			= \bigl| k_1 (A_1 \circ \sigma \circ \tau_2) \cdot (\sigma' \circ \tau_2) \cdot k_2 A_2 \bigr|
			= k_1 k_2 |\sigma' \circ \tau_2|,
	\]
	where we have used that by Lemma~\ref{lem:p2prop} the matrix $\xi'$ is a scalar
	multiple of an orthogonal matrix. Thus
	\begin{equation}\label{eq:cintau2}
		c(x) = \frac{K^2}{|\tau_2(x)|^2} \text{ with } K \coloneqq (k_1 k_2)^{1/2} > 0,
	\end{equation}
	compare also~\cite[Theorem~3.1.6]{Bea83}.
	Since $g \neq 0$, there exists a ball $B$ in $\Omega_2$ such that
	$g$ does not change sign in $B$. For simplicity we assume $g > 0$,
	the case $g < 0$ being similar.
	Then
	\[
		g(x) = c(x)^{\frac{N-2}{2}} = K^{N-2} |\tau_2(x)|^{2-N} = K^{N-2} k_2^{2-N} r^{2-N} \text{ on } B
	\]
	by Lemma~\ref{lem:p2prop}, where
	\begin{equation}\label{eq:rintau2}
		r \coloneqq \frac{|\tau_2(x)|}{k_2} = |x + z| \text{ for } z \coloneqq \frac{1}{k_2} A_2^{-1} y_2.
	\end{equation}
	Thus $g$ is a radial function with respect to the center $z$.
	Calculating $\Delta g$ in polar coordinates we obtain that
	\[
		\Delta g
			= g_{rr} + \frac{g_r}{r}
			= K^{N-2} k_2^{2-N} (2-N)^2 r^{-N}
			= \frac{(2-N)^2}{r^2}g
		\text{ on } B,
	\]
	hence
	\begin{equation}\label{eq:gforsphere}
		0 = g - \Delta g - gc^2
			= g \Bigl( 1 - \frac{(2-N)^2}{r^2} - \frac{K^4}{k_2^4 r^4} \Bigr) \text{ on } B
	\end{equation}
	by Lemma~\ref{lem:p2prop}, \eqref{eq:cintau2} and~\eqref{eq:rintau2}.
	Since $g$ has no zero, this implies
	\[
		r^4 - (2-N)^2 r^2 - K^4 k_2^{-4} = 0 \text{ for all } r \in S \coloneqq \biggl\{ \frac{|\tau_2(x)|}{k_2} : x \in B \biggr\},
	\]
	which contradicts the identity theorem for polynomials as $S$ possesses an interior point.
\end{proof}

Surprisingly, the conclusion $\nabla g = 0$ is false for $p=2$ and $N=1$.
\begin{example}\label{ex:p2N1}
	Let $\xi(y) \coloneqq -\Artanh \e^{-2y}$ and
	$g(y) \coloneqq \sinh^{1/2}(2y)$ on $\Omega_2 \coloneqq (1,2)$.
	Let $\Omega_1 \coloneqq \xi(\Omega_2)$ and
	$Tu \coloneqq g \cdot (u \circ \xi)$ for $u \in W^{1,2}(\Omega_1)$.
	Then $T$ satisfies~\href{ass:CoMu}, \href{ass:pLaplace}, and~\href{ass:W1p0}.
	In fact, $T$ is a lattice isomorphism from $W^{1,2}(\Omega_1)$
	to $W^{1,2}(\Omega_2)$ that satisfies $TW^{1,2}_0(\Omega_1) = W^{1,2}_0(\Omega_2)$.
\end{example}
\begin{proof}
	Note that $\xi'(y) = 1/\sinh(2y)$, so that $\xi$
	is strictly increasing on $\Omega_2$.
	Thus the only non-trivial statement is that $T$ satisfies~\href{ass:pLaplace}.

	Let $u \in W^{1,2}(\Omega_1)$ be arbitrary. We show that $g^4 (Tu - (Tu)'') = Tu - Tu''$.
	In fact,
	\begin{align*}
		g - g'' & = \frac{1}{g^3}, &
		\xi' & = \frac{1}{g^2}, &
		2g' \xi' & = \frac{2\cosh(2y)}{g^3} = -g \xi'',
	\end{align*}
	and hence
	\begin{align*}
		Tu - (Tu)'' & = g (u \circ \xi) - g'' (u \circ \xi) - 2 g' (u \circ \xi)' - g (u \circ \xi)'' \\
			& = (g - g'') (u \circ \xi) - 2 g' \xi' (u' \circ \xi) - g \xi'' (u' \circ \xi) \xi'' - g (u'' \circ \xi) (\xi')^2 \\
			& = \frac{1}{g^4} \bigl( g \cdot (u \circ \xi) - g \cdot (u'' \circ \xi) \bigr)
			= \frac{Tu - Tu''}{g^4}.
	\end{align*}
	Now let $v \in \mathcal{D}(\Omega_1)$ be arbitrary. Then
	\begin{align*}
		\mathfrak{a}_{2,\Omega_2}(Tu,Tv)
			& = \int_{\Omega_2} Tu \, Tv + \int_{\Omega_2} (Tu)' (Tv)'
			= \int_{\Omega_2} Tu \, \bigl( Tv - (Tv)'' \bigr) \\
			& = \int_{\Omega_2} \frac{1}{g^4} Tu \, \bigl( Tv - Tv'' \bigr)
			= \int_{\Omega_2} (u \circ \xi) \bigl( (v \circ \xi) - (v'' \circ \xi) \bigr) \xi' \\
			& = \int_{\xi(\Omega_2)} u \, (v - v'')
			= \int_{\Omega_1} uv + \int_{\Omega_1} u' v'
			= \mathfrak{a}_{2,\Omega_1}(u,v)
	\end{align*}
	by the substitution formula and partial integration.
	Since the test functions are dense, this proves~\href{ass:pLaplace}.
\end{proof}
It should be noted, however, that the operator $T$ in Example~\ref{ex:p2N1} is not isometric.
In fact, one can check (numerically) that
\[
	\| T\setone \|_{W^{1,2}(\Omega_2)}^2
		= \int_{\Omega_2} g^2 + \int_{\Omega_2} (g')^2
		\approx 23.66 \gg 0.118 \approx |\Omega_1| = \| \setone \|_{W^{1,2}(\Omega_1)}^2.
\]
But $T$ is isometric on $W^{1,2}_0(\Omega_1)$ as an immediate consequence of~\href{ass:pLaplace}
and the fact that $TW^{1,2}_0(\Omega_1) = W^{1,2}_0(\Omega_2)$.

\bigskip

We now turn our attention to the case $p > 2$.
\begin{proposition}\label{prop:nablanull}
	Let $p$ be in $(2,\infty)$, and let $T$ satisfy~\href{ass:CoMu}, \href{ass:pLaplace}, and~\href{ass:W1p0}.
	Then $\nabla g = 0$ on $\Omega_2$.
\end{proposition}

\begin{proof}
	By Proposition~\ref{prop:form}, Remark~\ref{rem:form2}, and assumption~\href{ass:pLaplace} we have that
	\begin{equation}\label{eq:binv}
		 \mathfrak{b}_{p,\Omega_1}(u,v,w) 
			= \mathfrak{b}_{p,\Omega_2}(Tu,Tv,Tw)
	\end{equation}
	for all $u,v \in W^{1,p}(\Omega_1)$ and $w \in W^{1,p}_0(\Omega_1)$
	satisfying $Tw \in W^{1,p}_c(\Omega_2)$.

	Let $f \in \mathrm{C}^\infty(\overline{\Omega}_1)$, and
	let $\eta \in \mathcal{D}(\Omega_2)$ be an arbitrary test function.
	By~\href{ass:W1p0} there exists $\phi$ in $W^{1,p}_0(\Omega_1)$ such that
	$T\phi = \eta \e^{-f\circ\xi} \in W^{1,p}_c(\Omega_2)$.
	Since $\e^f \phi$ is in $W^{1,p}_0(\Omega_1)$ and $T(\e^f \phi) = \eta$, we obtain
	from~\href{ass:pLaplace}, \eqref{eq:binv}, and the definitions of $\mathfrak{a}_{p,\Omega_1}$
	and $\mathfrak{b}_{p,\Omega_1}$ that
	\begin{equation}\label{eq:relationab}
		(p-1) \mathfrak{a}_{p,\Omega_2}(g, \eta)
			= (p-1) \mathfrak{a}_{p,\Omega_1}(\setone, \e^f \phi)
			= \mathfrak{b}_{p,\Omega_1}(\setone, \e^f, \phi)
			= \mathfrak{b}_{p,\Omega_2}(g, T\e^f, T\phi).
	\end{equation}
	We point out that that~\eqref{eq:relationab} is in general false for $p=2$.

	Expanding the left-most and right-most term of~\eqref{eq:relationab} by their definition and
	using $g \eta = T\e^f \; T\phi$ to cancel several terms, we arrive at
	\begin{equation}\label{eq:first}
		\begin{aligned}
			& (p-1) \int_{\Omega_2} |\nabla g|^{p-2} \scalar{\nabla g}{\nabla \eta} \\
				& \qquad = (p-2) \int_{\Omega_2} |\nabla g|^{p-4} \scalar{\nabla g}{\nabla (T\e^f)} \scalar{\nabla g}{\nabla (T\phi)} \\
				& \qquad\qquad + \int_{\Omega_2} |\nabla g|^{p-2} \scalar{\nabla (T\e^f)}{\nabla (T\phi)}.
		\end{aligned}
	\end{equation}
	Since
	\[
		\nabla\eta
			= \nabla(\e^{f \circ \xi} \cdot T\phi)
			= \e^{f\circ\xi} \, \nabla(f \circ \xi) \, T\phi + \e^{f\circ\xi} \, \nabla(T\phi)
			= \eta \, \nabla(f \circ \xi) + \e^{f\circ\xi} \, \nabla(T\phi),
	\]
	we have
	\[
		\nabla(T\phi) = \e^{-f\circ\xi} \bigl( \nabla\eta - \eta \, \nabla(f\circ\xi) \bigr).
	\]
	Plugging the latter identity and
	\[
		\nabla(T\e^f)
			= \nabla(g \cdot \e^{f\circ\xi})
			= \e^{f\circ\xi} \bigl(\nabla g + g \, \nabla(f\circ\xi)\bigr)
	\]
	into the right hand side of~\eqref{eq:first}, we obtain
	\begin{align*}
		& (p-1) \int_{\Omega_2} |\nabla g|^{p-2} \scalar{\nabla g}{\nabla\eta} \\
			& \qquad = (p-2) \int_{\Omega_2} |\nabla g|^{p-4} \scalar{\nabla g}{\nabla g + g \nabla(f\circ\xi)} \scalar{\nabla g}{\nabla\eta - \eta \, \nabla(f \circ \xi)} \\
			& \qquad\qquad + \int_{\Omega_2} |\nabla g|^{p-2} \scalar{\nabla g + g \nabla(f\circ\xi)}{\nabla\eta - \eta \, \nabla(f\circ\xi)}.
	\end{align*}

	Adding the latter equation to the corresponding one with $f$ replaced by $-f$ and multiplying out, several terms
	cancel. Dividing by $2$, we end up with
	\begin{align*}
		0 = (p-2) \int_{\Omega_2} |\nabla g|^{p-4} g \eta \, \scalar{\nabla g}{\nabla (f\circ\xi)}^2
			+ \int_{\Omega_2} |\nabla g|^{p-2} g \eta \, |\nabla (f\circ\xi)|^2.
	\end{align*}
	Since $\eta$ is an arbitrary test function,
	\[
		0 = (p-2) |\nabla g|^{p-4} g \, \scalar{\nabla g}{\nabla(f\circ\xi)}^2 + |\nabla g|^{p-2} g \, |\nabla(f\circ\xi)|^2
	\]
	almost everywhere. Since both summands have the same sign,
	\begin{equation}\label{eq:vanish}
		|\nabla g|^{p-2} g \, |\nabla (f\circ\xi)|^2 = 0.
	\end{equation}

	Recall that we want to show that $\nabla g = 0$. Assume to the contrary
	that $|\nabla g| > 0$ on a set $U$ of positive measure, which is necessarily open
	due to the continuity of $\nabla g$, and which we can choose
	to be connected. Then~\eqref{eq:vanish} implies that $\nabla (f\circ\xi) = 0$
	on $U$ for every smooth function $f$.
	In particular $\nabla \xi_j = 0$ on $U$ for each $j$,
	i.e., $\xi$ is constant on $U$. But this contradicts~\href{ass:CoMu}.
\end{proof}

So if either $p=2$ and $N \ge 2$ or $p > 2$ we know that $\nabla g = 0$.
In the remainder of this section we will take this as an assumption instead
of conditions on $p$ and $N$.
\begin{proposition}\label{prop:ortho}
	Let $T$ satisfy~\href{ass:CoMu}, \href{ass:pLaplace}, and~\href{ass:W1p0}.
	Assume that $\nabla g = 0$.
	Then $\xi'(x)$ is an orthogonal matrix for all $x \in \Omega_2$.
\end{proposition}

\begin{proof}
	Let $u$ be as in Lemma~\ref{lem:harm}, i.e.,
	\[
		u = \e^f \text{ for } f(x) = \sum_{i=1}^N \alpha_i x_i
			\text{ with } \sum_{i=1}^N \alpha_i^2 = (p-1)^{-2/p},
	\]
	so that $\Delta_p u = u^{p-1}$.
	Let $\eta \in \mathcal{D}(\Omega_2)$ be an arbitrary test function.
	By~\href{ass:W1p0} there exists $\phi \in W^{1,p}_0(\Omega_1)$ such that $T\phi = \e^{-(p-1) (f \circ \xi)} \eta \in W^{1,p}_c(\Omega_2)$.
	Now
	\begin{equation}\label{eq:psiphirel2}
		\begin{aligned}
			\nabla\eta
				& = (p-1) (u\circ\xi)^{p-1} \nabla(f\circ\xi) \, T\phi + (u\circ\xi)^{p-1} \nabla(T\phi) \\
				& = (p-1) \nabla(f\circ\xi) \, \eta + (u\circ\xi)^{p-1} \nabla(T\phi),
		\end{aligned}
	\end{equation}
	where we used $\nabla(u\circ\xi) = (u\circ\xi) \nabla(f\circ\xi)$.
	Now by~\href{ass:pLaplace}, since $\nabla g = 0$,
	\begin{align*}
		0 & = \mathfrak{a}_{p,\Omega_1}(u,\phi)
			= \mathfrak{a}_{p,\Omega_2}(Tu,T\phi) \\
			& = \int_{\Omega_2} |g|^{p-2} g (u\circ\xi)^{p-1} T\phi + \int_{\Omega_2} |g|^{p-2} g (u\circ\xi)^{p-1} |\nabla (f\circ\xi)|^{p-2} \nabla(f\circ\xi) \nabla(T\phi) \\
			& = \int_{\Omega_2} |g|^{p-2} g \eta + \int_{\Omega_2} |g|^{p-2} g |\nabla (f\circ\xi)|^{p-2} \, \nabla(f\circ\xi) \nabla\eta \\
			& \qquad - (p-1) \int_{\Omega_2} |g|^{p-2} g |\nabla (f\circ\xi)|^p \, \eta,
	\end{align*}
	where for the last identity we used~\eqref{eq:psiphirel2}.
	Adding the last equation to the corresponding one with
	$f$ replaced by $-f$ and dividing by $2$, we have shown that
	\[
		0 = \int_{\Omega_2} |g|^{p-2} g \eta - (p-1) \int_{\Omega_2} |g|^{p-2} g |\nabla(f\circ\xi)|^p \eta.
	\]
	Since $\eta$ was an arbitrary test function, $|\nabla(f\circ\xi)| \equiv (p-1)^{-1/p}$ on $\Omega_2$.

	Putting in particular $f(x) \coloneqq (p-1)^{-1/p} x_i$ for some $i \in \{1,\dots,N\}$
	shows that $|\nabla\xi_i| \equiv 1$ on $\Omega_2$.
	But now taking $f(x) = \alpha x_i + \beta x_j$ for $i \neq j$ such that
	$\alpha^2 + \beta^2 = (p-1)^{-2/p}$ and $\alpha, \beta \neq 0$ we obtain
	\begin{align*}
		(p-1)^{-2/p}
			& = |\nabla(f\circ\xi)|^2
			= |\alpha \nabla\xi_i + \beta \nabla\xi_j|^2 \\
			& = \alpha^2 + \alpha \beta \scalar{\nabla\xi_i}{\nabla\xi_j} + \beta^2
			= (p-1)^{-2/p} + \alpha \beta \scalar{\nabla\xi_i}{\nabla\xi_j}.
	\end{align*}
	This shows that $\nabla\xi_i$ and $\nabla\xi_j$ are orthogonal at every point of $\Omega_2$.
\end{proof}

The following is a consequence of preceding theorem, see~\cite[Proposition~2.3]{Are02}.
\begin{corollary}\label{cor:lociso}
	Let $T$ satisfy~\href{ass:CoMu}, \href{ass:pLaplace}, and~\href{ass:W1p0}.
	Assume that $\nabla g = 0$.
	Then $\xi$ is locally a rigid motion.
	In other words, the restriction of $\xi$ to any connected component of $\Omega_2$
	is a rigid motion. If moreover $\Omega_2$ is connected, then
	$\xi$ extends to a rigid motion on $\mathds{R}^N$.
\end{corollary}

\begin{proposition}\label{prop:essbij}
	Let $T$ satisfy~\href{ass:CoMu}, \href{ass:pLaplace}, and~\href{ass:W1p0}.
	Assume that $\nabla g = 0$.
	Then $\xi$ is bijective from the open set $U_2 \coloneqq \xi^{-1}(\Omega_1)$
	to the open subset $U_1 \coloneqq \xi(U_2)$ of $\Omega_1$.
	Moreover, $N_2 \coloneqq \Omega_2 \setminus U_2$ has Lebesgue measure zero
	and $U_1$ is dense in $\overline{\Omega}_1$ (equivalently:
	$N_1 \coloneqq \Omega_1 \setminus U_1$ has no interior points).
\end{proposition}

\begin{proof}
	The sets $U_1$ and $U_2$ are open since $\xi$ is a continuous open mapping by Corollary~\ref{cor:lociso}.
	By construction $\xi$ is surjective from $U_2$ to $U_1$.
	Since $\xi(x) \in \Omega_1$ for almost every $x \in \Omega_2$, $N_2$ has measure zero.

	Assume for contradiction that $\xi$ is not injective from $U_2$ to $U_1$, i.e., there exist
	$y_1 \neq y_2$ in $U_2$ such that $x \coloneqq \xi(y_1) = \xi(y_2)$.
	By Corollary~\ref{cor:lociso} there exists $r > 0$ such that
	$\xi$ is a rigid motion from $B_i \coloneqq B(y_i,r) \subset U_2$
	to $B_0 \coloneqq B(x,r) \subset U_1$,
	$B_1 \cap B_2 = \emptyset$.
	Consider a test function $\phi \in \mathcal{D}(\Omega_2)$, $\phi \neq 0$, such that
	$\supp\phi \subset B_1$. By~\href{ass:W1p0}
	there exists $u \in W^{1,p}_0(\Omega_1)$ such that $Tu = g \cdot (u\circ\xi) = \phi$
	almost everywhere. Since $g$ has no zero and $\phi=0$ on $B_2$,
	$u$ has to vanish almost everywhere on $\xi(B_2) = B_0$.
	But then $\phi$ vanishes almost everywhere on $B_1 = \xi^{-1}(B_0)$,
	a contradiction.
	Thus we have shown that $\xi$ is bijective from $U_2$ to $U_1$.

	Assume for contradiction that $N_1$ has an interior point, i.e.,
	there exists an open set $V \subset N_1$.
	Let $\phi \in \mathcal{D}(V)$, $\phi \neq 0$. Then $T\phi = 0$, thus by~\href{ass:pLaplace}
	\[
		0 < \int_{\Omega_1} |\phi|^p + \int_{\Omega_1} |\nabla \phi|^p
			= \mathfrak{a}_{p,\Omega_1}(\phi,\phi)
			= \mathfrak{a}_{p,\Omega_2}(T\phi,T\phi)
			= 0,
	\]
	a contradiction. Thus $N_1$ has no interior points.
\end{proof}
It is obvious that the choice of $U_1$ and $U_2$
in Proposition~\ref{prop:essbij} is maximal, i.e., that every
other pair of sets $V_1 \subset \Omega_1$ and $V_2 \subset \Omega_2$
such that $\xi$ is bijective from $V_2$ to $V_1$ has the property
that $V_i \subset U_i$ for $i=1,2$.

\begin{corollary}
	Let $T$ satisfy~\href{ass:CoMu}, \href{ass:pLaplace}, and~\href{ass:W1p0}.
	Assume that $\nabla g = 0$. Then $|g| \equiv 1$ on $\Omega_2$.
\end{corollary}
\begin{proof}
	Let $U_1$ and $U_2$ be as in Proposition~\ref{prop:essbij}, so that
	$\xi$ is a diffeomorphism from $U_2$ to $U_1$.
	Let $v \in \mathcal{D}(U_2)$ be arbitrary. Then there exists $u \in \mathcal{D}(U_1)$
	such that $u \circ \xi = v$.
	Thus by~\href{ass:pLaplace} and the substitution rule
	\begin{equation}\label{eq:substxi}
		\int_{U_2} |g|^p v
			= \mathfrak{a}_{p,\Omega_2}(g, Tu)
			= \mathfrak{a}_{p,\Omega_1}(\setone, u)
			= \int_{U_1} u
			= \int_{U_2} v
	\end{equation}
	since $|\det\xi'| \equiv 1$.
	As the test functions are dense, from this we see that
	$|g| \equiv 1$ on $U_2$. But since $g$ is continuous and
	$U_2$ is dense in $\Omega_2$, $|g| \equiv 1$ on $\Omega_2$.
\end{proof}

Under the assumptions of this section, the set $N_1$
in Proposition~\ref{prop:essbij} may be rather large in measure
as Example~\ref{ex:largeN1} shows.
On the other hand, in Proposition~\ref{prop:smallN1} we show that
for isometries this cannot be the case.
Conditions under which $N_1$ is empty will
be discussed in the next section.

\begin{example}\label{ex:largeN1}
	Let $\Omega_1 \coloneqq (0,1)$, and let $\Omega_2$ be an open
	set of Lebesgue measure $|\Omega_2| < 1$ such that $\overline{\Omega}_2 = [0,1]$.
	Let $Tu \coloneqq u \circ \xi$ for $\xi \coloneqq \id|_{\Omega_2}$.
	Then $T$ satisfies~\href{ass:CoMu}, \href{ass:pLaplace}, and~\href{ass:W1p0}.
	With the notations of Proposition~\ref{prop:essbij}
	we have $N_2 = \emptyset$ and $N_1 = (0,1) \setminus \Omega_1$.
	Thus $N_1$ has positive measure.
\end{example}

\begin{proposition}\label{prop:smallN1}
	Let $T$ be an isometric operator from $W^{1,p}(\Omega_1)$
	to $W^{1,p}(\Omega_2)$ that satisfies~\href{ass:CoMu} and~\href{ass:W1p0}.
	Then the set $N_1$ in Proposition~\ref{prop:essbij} has Lebesgue measure zero.
\end{proposition}
\begin{proof}
	By Proposition~\ref{prop:essbij}, the Lebesgue measures of $U_1$, $U_2$,
	and $\Omega_2$ agree. If $T$ is isometric, then in particular
	\[
		|\Omega_2|^p
			= \| \setone_{\Omega_2} \|_{W^{1,p}(\Omega_2)}
			= \| T\setone_{\Omega_1} \|_{W^{1,p}(\Omega_2)}
			= \| \setone_{\Omega_1} \|_{W^{1,p}(\Omega_1)}
			= |\Omega_1|^p.
	\]
	Hence $|U_1| = |\Omega_1|$, which shows that $N_1$ has measure zero.
\end{proof}

For convenience, we summarize the main results of this section in the following theorem,
confer also Theorem~\ref{thm:HomCoMu}.
\begin{theorem}\label{thm:sumrep}
	Let $\Omega_1$ and $\Omega_2$ be bounded open subsets of $\mathds{R}^N$, $N \ge 1$,
	$p \in (1,\infty)$, and $T$ be an operator from $W^{1,p}(\Omega_1)$
	to $W^{1,p}_{\mathrm{loc}}(\Omega_2)$ such that $W^{1,p}_0(\Omega_2) \subset TW^{1,p}_0(\Omega_1)$.
	Assume any of the following premises:
	\begin{enumerate}[(a)]
	\item
		$T$ satisfies~\href{ass:CoMu} and~\href{ass:pLaplace}, and the $g$ in~\href{ass:CoMu}
		is locally constant; or
	\item
		$p = 2$, $N \ge 2$, $T$ is a weighted composition operator, and~\href{ass:pLaplace} holds; or
	\item
		$p = 2$, $N \ge 2$, $T$ is a weighted composition operator, and $T$ is isometric from
		$W^{1,p}(\Omega_1)$ to $W^{1,p}(\Omega_2)$; or
	\item
		$p > 2$, $T$ is a weighted composition operator, and~\href{ass:pLaplace} holds; or
	\item
		$p > 2$, $T$ is a weighted composition operator, and $T$ is isometric from
		$W^{1,p}(\Omega_1)$ to $W^{1,p}(\Omega_2)$; or
	\item
		$p > 2$, $T$ is an isometric operator from $W^{1,p}(\Omega_1)$ to $W^{1,p}(\Omega_2)$,
		is order bounded from $W^{1,p}(\Omega_1)$ to $L^p(\Omega_2)$, and satisfies $T(0) = 0$.
	\end{enumerate}
	Then there exists a local isometry $\xi$ from $\Omega_2$ to $\overline{\Omega}_1$ with
	dense image, which is injective apart from a set of measure zero,
	and a locally constant function $g$, which satisfies $|g| = 1$, such that
	$Tu = g \cdot (u \circ \xi)$ almost everywhere for each $u \in W^{1,p}(\Omega_1)$.
\end{theorem}

\section{Congruence of the Domains and Capacitary Results}\label{sec:cong}
Let $\Omega_1$ and $\Omega_2$ be bounded open subsets of $\mathds{R}^N$.
In this section we only consider operators $T$ from $W^{1,p}(\Omega_1)$ to $W^{1,p}(\Omega_2)$
such that
\begin{equation}\label{eq:Trig}
	\begin{aligned}
		Tu & = g \cdot (u \circ \xi) \text{ almost everywhere for all $u \in W^{1,p}(\Omega_1)$}, \\
			& \quad \text{where $g$ is locally constant, $|g| \equiv 1$, and $\xi$ is locally a rigid motion}.
	\end{aligned}
\end{equation}
Sufficient conditions for~\eqref{eq:Trig} to hold are given in Theorem~\ref{thm:sumrep}.
Moreover, we will usually assume that~\href{ass:W1p0} is satisfied, but mention this
explicitly.

We now look for conditions that guarantee that $\Omega_1$ and $\Omega_2$ are congruent.
Since a local rigid motion is a rigid motion on every connected component,
for connected $\Omega_2$ this is the same
as asking whether $\xi$ is bijective from $\Omega_1$ to $\Omega_2$, i.e., whether
the sets $N_1$ and $N_2$ of Proposition~\ref{prop:essbij} are empty.
To obtain the optimal results in this direction, we use the $p$-capacity $\Capa_p$ to measure
the size of the defects $N_1$ and $N_2$.

\begin{lemma}\label{lem:repquasi}
	Let $T$ be a weighted composition operator from $W^{1,p}(\Omega_1)$ to $W^{1,p}(\Omega_2)$,
	i.e., $Tu = g \cdot (u \circ \xi)$ almost everywhere for all $u \in W^{1,p}(\Omega_1)$.
	Assume that $g$ and $\xi$ are continuous.
	Then for every $p$-quasi continuous $u$ in $W^{1,p}_0(\Omega_1)$ we have
	$Tu = g \cdot (u \circ \xi)$ $p$-quasi everywhere
	if we take the $p$-quasi continuous representative of $Tu$.
	Here we set $u \coloneqq 0$ on $\partial\Omega_1$.
\end{lemma}
\begin{proof}
	Let $\mathscr{W}^{1,p}(\Omega_1)$ denote the closure of
	$W^{1,p}(\Omega_1) \cap \mathrm{C}(\overline{\Omega})$ in $W^{1,p}(\Omega_1)$.
	There exist $\zeta$ and $h$ such that
	$Tu = h \cdot (u \circ \zeta)$ $p$-quasi everywhere for the $p$-quasi continuous
	representatives whenever $u \in \mathscr{W}^{1,p}(\Omega_1)$~\cite[Theorem~4.5]{Bie08f},
	in particular for $u$ in $W^{1,p}_0(\Omega_1)$.
	Note that in~\cite{Bie08f} functions in $\mathscr{W}^{1,p}(\Omega_1)$
	are defined even on $\overline{\Omega}_1$ and we can set $u = 0$ on
	$\partial\Omega_1$ for $u \in W^{1,p}_0(\Omega_1)$.

	It suffices to show that $h = g$ and $\xi = \zeta$ $p$-quasi everywhere.
	For this, consider $u_0(x) \coloneqq 1$ and $u_j(x) \coloneqq x_j$, which are
	continuous functions in $\mathscr{W}^{1,p}(\Omega_1)$. Since $g$
	is a continuous representative of $Tu_0$, $g = h \cdot (u_0 \circ \zeta) = h$
	$p$-quasi everywhere.
	Since also $g \, \xi_j$ is continuous,
	$g \, \xi_j = Tu_j = h \cdot (u_j \circ \zeta) = h \, \zeta_j$
	$p$-quasi everywhere for every $j$. Since $g=h$ $p$-quasi everywhere,
	we can divide by $g$ and obtain that $\xi_j = \zeta_j$ $p$-quasi everywhere,
	i.e., $\xi = \zeta$ $p$-quasi everywhere.
\end{proof}


\begin{proposition}\label{prop:N2polar}
	Let $T$ satisfy~\eqref{eq:Trig} and~\href{ass:W1p0}.
	Then the set $N_2$ of Proposition~\ref{prop:essbij} is $p$-polar.
\end{proposition}
\begin{proof}
	Set $U_2 \coloneqq \Omega_2 \setminus N_2$.
	Let $v \in W^{1,p}_0(\Omega_2)$ be $p$-quasi continuous.
	By~\href{ass:W1p0} and Lemma~\ref{lem:repquasi} we have
	$v = g \cdot (u \circ \xi)$ $p$-quasi everywhere for some $p$-quasi continuous
	function $u \in W^{1,p}_0(\Omega_1)$, so in particular
	$v = 0$ $p$-quasi everywhere on $N_2$.
	Thus the restriction of $v$ to $U_2$ is in
	$W^{1,p}_0(U_2)$~\cite[Lemma~2.26]{MZ97}.
	Hence $N_2 = \Omega_2 \setminus U_2$ is $p$-polar~\cite[Theorem~2.15]{MZ97}.
%
\end{proof}

An almost trivial example shows that we cannot expect $N_1$ to be $p$-polar, even
if $\Omega_1$ and $\Omega_2$ are very smooth.
Here we give a counterexample with Lipschitz regular sets. A very similar,
but more technical construction would provide a counterexample with $\mathrm{C}^\infty$-regular sets.
Note, however, that in this example $\Omega_2$ is not connected and
$W^{1,p}_0(\Omega_2) \subsetneqq TW^{1,p}_0(\Omega_1)$.
\begin{example}\label{ex:disconn}
	Let $\Omega_1 \coloneqq (0,1) \times (-1,1)$ and $\Omega_2 \coloneqq (0,1) \times \bigl( (-2,-1) \cup (1,2) \bigr)$.
	Define $Tu \coloneqq u \circ \xi$ for $\xi(x,y) \coloneqq (x, y - \sgn y)$.
	Then $T$ is an isometric lattice homomorphism from $W^{1,p}(\Omega_1)$ to $W^{1,p}(\Omega_2)$
	such that $W^{1,p}_0(\Omega_2) \subset TW^{1,p}_0(\Omega_1)$ for every $p \in (1,\infty)$.
	Thus $T$ satisfies~\href{ass:CoMu}, \href{ass:pLaplace}, \href{ass:W1p0}, and~\eqref{eq:Trig}.
	The sets of Proposition~\ref{prop:essbij} are in this case
	$N_2 = \emptyset$ and $N_1 = (0,1) \times \{0\}$. So $N_1$ is not $p$-polar
	for any $p > 1$.
\end{example}

On the other hand, the combination of some mild regularity assumptions and connectedness
suffices for the defects $N_1$ and $N_2$ to be empty. An appropriate regularity condition is the following,
compare also Remark~\ref{rem:regcap}.
\begin{definition}\label{def:regcap}
	An open set $\Omega \subset \mathds{R}^N$ is called \emph{regular in $p$-capacity}
	if for every $z \in \partial\Omega$ and every $r > 0$ the set
	$B(z,r) \setminus \Omega$ has positive $p$-capacity, i.e., is not $p$-polar.
\end{definition}
\begin{remark}\label{rem:topreg}
	Regularity in $p$-capacity is not a restrictive assumption.
	For example, if a set is topologically regular, i.e., if $\Omega$ is
	the interior of $\overline{\Omega}$, then $\Omega$ is regular in
	in $p$-capacity for each $p > 1$. In particular smooth domains, e.g.\ domains
	with continuous boundary, are topologically regular and hence regular in $p$-capacity.
\end{remark}

\begin{proposition}\label{prop:regcap}
	Let $T$ satisfy~\eqref{eq:Trig} and~\href{ass:W1p0}.
	If $\Omega_1$ is regular in $p$-capacity,
	then the set $N_2$, defined as in Proposition~\ref{prop:essbij}, is empty.
\end{proposition}
\begin{proof}
	Assume to the contrary that there exists $y_0 \in \Omega_2$ such that $\xi(y_0) \in \partial\Omega_1$.
	Since $\xi$ is open, there exists $r > 0$ such that
	$B(\xi(y_0),r) \subset \xi(\Omega_2)$. Hence $\xi(\Omega_2) \setminus \Omega_1$
	has positive $p$-capacity by assumption, contradicting the fact that by Proposition~\ref{prop:N2polar}
	$\xi(y) \in \Omega_1$ for $p$-quasi every $y \in \Omega_2$.
\end{proof}

\begin{theorem}\label{thm:conn}
	Let $T$ satisfy~\eqref{eq:Trig} and~\href{ass:W1p0}.
	Let $\Omega_1$ be regular in $p$-capacity, and let $\Omega_2$ be topologically regular and connected.
	Then $\xi$ is bijective from $\Omega_2$ to $\Omega_1$, and the sets $\Omega_1$ and $\Omega_2$ are congruent.
\end{theorem}
\begin{proof}
	Since $\xi$ is a rigid motion, it can be
	extended to an isometric automorphism of $\mathds{R}^N$,
	which we again denote by $\xi$.
	As in the proof of Proposition~\ref{prop:essbij}, $U_1 \coloneqq \xi(\Omega_2)$
	is dense in $\overline{\Omega}_1$.
	By Proposition~\ref{prop:regcap}, $U_1 \subset \Omega_1$.
	Since $U_1$ and $\Omega_2$ are congruent,
	we only have to show that $U_1 = \Omega_1$.

	Since $\xi$ maps $\Omega_2$ onto $U_1$, $\xi$ maps $\overline{\Omega}_2$ onto $\overline{\Omega}_1$.
	Since $\xi$ is an homeomorphism of $\mathds{R}^N$,
	it maps the interior of $\overline{\Omega}_2$, which is $\Omega_2$,
	onto the interior of $\overline{\Omega}_1$, which contains $\Omega_1$.
	Thus $\Omega_1 \subset U_1 = \xi(\Omega_2)$, i.e., $U_1 = \Omega_1$.
\end{proof}

The statement of Theorem~\ref{thm:conn} becomes false if we assume $\Omega_2$ only to
be regular in $p$-capacity, even if $T$ is an isometric isomorphism from $W^{1,p}(\Omega_1)$
to $W^{1,p}(\Omega_2)$. However, the counterexample that we will give now is more
involved than Example~\ref{ex:disconn}.
\begin{example}
	Let $N \ge 2$ and $1 < p < N$. Let $M \subset \mathds{R}^N$ be a compact set of Hausdorff dimension
	$s \in (N-p, N-1)$, and assume without loss of generality that $M \subset (0,1)^N$.
	The existence of such a set can be obtained by the theory of fractals.
	In fact, let $s \in (N-p, N-1)$ be arbitrary. Pick $k$ so large that $k$
	disjoint copies $(B_i)_{i=1}^k$ of an $N$-dimensional cube of edge length $k^{-1/s}$ fit into the unit cube,
	and let $\phi_i$ denote the affine linear mapping that maps the unit cube onto $B_i$.
	Then the unique non-empty, compact, $(\phi_i)_{i=1}^k$-invariant
	set $M$ has Hausdorff dimension $s$~\cite[Theorem~6.5.4]{Edg08}.

	Define $\Omega_1 \coloneqq (0,1)^N$ and $V_2 \coloneqq (0,1)^N \setminus M$.
	Denote by $\Omega_2$ the open superset of $V_2$ which is regular in $p$-capacity
	and satisfies $\Capa_p(\Omega_2 \triangle V_2) = 0$~\cite[Proposition~3.2.6]{Bie05}.
	Then $(0,1)^N \setminus \Omega_2$ is a subset of $M$, hence has zero
	$(N-1)$-dimensional Hausdorff measure. Thus $\Omega_2$ is connected by Lemma~\ref{lem:capconn}.
	Moreover, $\Omega_2 \neq (0,1)^N$ since $M$ is not $p$-polar~\cite[Theorem~2.53]{MZ97}.

	Finally, let $Tu \coloneqq u \circ \xi$ for $\xi \coloneqq \id|_{\Omega_2}$.
	Then $T$ is an operator from $W^{1,p}(\Omega_1)$ to $W^{1,p}(\Omega_2)$.
	Since $M$ is closed with zero $(N-1)$-dimensional Hausdorff measure, the restriction to $V_2$
	is an isometric isomorphism from $W^{1,p}( (0,1)^N )$ to $W^{1,p}(V_2)$~\cite[\S 1.2.5]{MP97}.
	Since $V_2 \subset \Omega_2 \subset (0,1)^N$, the restriction to $\Omega_2$
	is an isometric isomorphism from $W^{1,p}( (0,1)^N )$ to $W^{1,p}(\Omega_2)$.
	Hence $T$ is an isometric isomorphism from $W^{1,p}(\Omega_1)$ to $W^{1,p}(\Omega_2)$
	such that $W^{1,p}_0(\Omega_2) \subset W^{1,p}_0( (0,1)^N ) = TW^{1,p}_0(\Omega_1)$.
	Moreover, $T$ satisfies~\href{ass:pLaplace} since $\Omega_1$ and $\Omega_2$ differ only
	by a set of Lebesgue measure zero.

	Thus $T$ satisfies~\href{ass:CoMu}, \href{ass:pLaplace}, \href{ass:W1p0}, and~\eqref{eq:Trig}.
	The sets $\Omega_1$ and $\Omega_2$ are regular in $p$-capacity by construction. However, the
	set $N_1$ of Proposition~\ref{prop:essbij}, which is $N_1 = (0,1)^N \setminus \Omega_2$, is not $p$-polar.
\end{example}

On the other hand, if we have some more information about $T$ as an operator, then
regularity in $p$-capacity suffices for $\Omega_2$ as well.
\begin{theorem}\label{thm:W1p0eq}
	Let $T$ satisfy~\eqref{eq:Trig}.
	Let $\Omega_1$ and $\Omega_2$ be regular in $p$-capacity.
	Assume that $T$ is bijective from $W^{1,p}_0(\Omega_1)$ onto $W^{1,p}_0(\Omega_2)$.
	Then $\xi$ is bijective from $\Omega_2$ onto $\Omega_1$.
\end{theorem}
\begin{proof}
	Let $U_1 \coloneqq \xi(\Omega_2)$. Then $U_1$ is a subset of $\Omega_1$ by Proposition~\ref{prop:regcap},
	so $W^{1,p}_0(U_1)$ is in a natural way a subspace of $W^{1,p}(\Omega_1)$.
	Since $\xi$ is a smooth diffeomorphism from $\Omega_2$ to $U_1$,
	$T$ maps $\mathcal{D}(U_1)$ onto $\mathcal{D}(\Omega_2)$, so it maps $W^{1,p}_0(U_1)$ isometrically
	onto $W^{1,p}_0(\Omega_2)$.
	Let $\phi \in W^{1,p}_0(\Omega_1)$ be arbitrary. Then $T\phi \in W^{1,p}_0(\Omega_2)$ by assumption.
	Hence there exists $\tilde{\phi}$ in $W^{1,p}_0(U_1)$, which is in a natural way a subset of
	$W^{1,p}_0(\Omega_1)$, such that
	$T\tilde{\phi} = T\phi$. But then we have even
	$\phi = \tilde{\phi} \in W^{1,p}_0(U_1)$ since $T$ is injective.
	This shows $W^{1,p}_0(\Omega_1) = W^{1,p}_0(U_1)$ as subspaces of $W^{1,p}(\mathds{R}^N)$,
	hence $U_1 = \Omega_1$ by~\cite[Theorem~3.2.15]{Bie05}.
\end{proof}

\section{Applications}\label{sec:app}

In this section we collect our results into some theorems that
treat several typical situations. Different versions of these theorems
can easily be constructed from the previous results.

For all the theorems, $\Omega_1$ and $\Omega_2$ are bounded open subsets of
$\mathds{R}^N$, $p$ is in $(1,\infty)$, and $T$ is an operator from $W^{1,p}(\Omega_1)$ to
$W^{1,p}(\Omega_2)$.

\begin{theorem}\label{thm:isorig}
	Let $p=2$ and $N \ge 2$, or let $p > 2$.
	Let $T$ be an isometric lattice homomorphism such that $W^{1,p}_0(\Omega_2) \subset TW^{1,p}_0(\Omega_1)$.
	Then there exists a local rigid motion $\xi$ on $\Omega_2$ such that
	$Tu = u \circ \xi$ holds almost everywhere for all $u \in W^{1,p}(\Omega_1)$.
\end{theorem}
\begin{proof}
	Every lattice homomorphism is a weighted composition operator by~\cite[Theorem~4.13]{Bie08f}.
	Then $Tu = g \cdot (u \circ \xi)$ almost everywhere for all $u \in W^{1,p}(\Omega_1)$
	with a local rigid motion $\xi$ and a locally constant
	function $g$ such that $|g| \equiv 1$ by Theorem~\ref{thm:sumrep}.
	But since $T$ is a positive operator, $g \equiv 1$.
\end{proof}

\begin{theorem}\label{thm:orderrig}
	Let $p > 2$, and let $T$ be isometric. Assume that $T$ is order bounded from
	$W^{1,p}(\Omega_1)$ to $L^p(\Omega_2)$ and that $W^{1,p}_0(\Omega_2) \subset TW^{1,p}_0(\Omega_1)$
	and $T(0) = 0$.
	Then there exists a local rigid motion $\xi$ on $\Omega_2$ and a locally constant function $g$
	on $\Omega_2$ satisfying $|g| \equiv 1$ such that
	$Tu = g \cdot (u \circ \xi)$ holds almost everywhere for all $u \in W^{1,p}(\Omega_1)$.
\end{theorem}
\begin{proof}
	The operator $T$ is linear~\cite{Vae03}.
	Then, by Theorem~\ref{thm:HomCoMu}, $T$ is a weighted composition operator. Now the claim
	follows from Theorem~\ref{thm:sumrep}.
\end{proof}

The remaining theorems of this section give conditions under which the sets $\Omega_1$ and $\Omega_2$ are
congruent. For this we always assume
some mild regularity of the sets, at least regularity in $p$-capacity, see Definition~\ref{def:regcap}.
The following remark explains why this is optimal.
Moreover, it should be noted that this assumptions is very weak, compare Remark~\ref{rem:topreg}.
\begin{remark}\label{rem:regcap}
	Sobolev spaces ``do not see'' very small differences of open sets in the sense that
	$W^{1,p}(\Omega)$ and $W^{1,p}_0(\Omega)$ are by the
	identity mapping isometrically isomorphic to $W^{1,p}(\Omega')$ and $W^{1,p}_0(\Omega')$,
	respectively, whenever $\Omega$ and $\Omega'$ differ only by a $p$-polar set.
	This follows for example from~\cite[Theorem~2.53]{MZ97}, \cite[\S 1.2.5]{MP97},
	and~\cite[Theorem~3.2.15]{Bie05}.

	To give an example, for $1 < p \le N$ we pick $\Omega$ and $\Omega'$ such that they
	differ by a single point. If we let $T$ be the identity operator
	from $W^{1,p}(\Omega)$ to $W^{1,p}(\Omega')$, then $T$ has all nice (operator) properties
	we can ask for, e.g., $T$ is an isometric lattice isomorphism,
	but the domains are not congruent.

	One way to exclude such artificial counterexamples is to pick canonical
	representatives for the equivalence classes of domains that differ only by $p$-polar
	sets. A way to express this choice is to require the domain to be regular
	in $p$-capacity. In fact, for every open set $\Omega$
	there exists a unique open set which is regular in $p$-capacity and differs from
	$\Omega$ only by a $p$-polar set~\cite[Proposition~3.2.6]{Bie05}.
\end{remark}

\begin{theorem}
	Let $p = 2$ and $N \ge 2$, or let $p > 2$.
	Let $\Omega_1$ be regular in $p$-capacity and $\Omega_2$ be connected and topologically regular, i.e.,
	$\Omega_2$ is the interior of its closure.
	Assume that there exists an isometric operator $T$ such that
	$W^{1,p}_0(\Omega_2) \subset TW^{1,p}_0(\Omega_1)$
	and either
	\begin{enumerate}[(a)]
	\item
		$T$ is a lattice homomorphism; or
	\item
		$p > 2$, $T$ is order bounded from $W^{1,p}(\Omega_1)$ to $L^p(\Omega_2)$, and $T(0)=0$.
	\end{enumerate}
	Then $\Omega_1$ and $\Omega_2$ are congruent.
\end{theorem}
\begin{proof}
	By Theorems~\ref{thm:isorig} and~\ref{thm:orderrig} the operator $T$ satisfies~\eqref{eq:Trig}.
	Now the claim follows from Theorem~\ref{thm:conn}.
\end{proof}

\begin{theorem}
	Let $p=2$ and $N \ge 2$, or let $p > 2$.
	Assume that $\Omega_1$ and $\Omega_2$ be regular in $p$-capacity.
	Assume moreover that there exists an isometric operator $T$ such that
	$TW^{1,p}_0(\Omega_1) = W^{1,p}_0(\Omega_2)$
	and either
	\begin{enumerate}[(a)]
	\item
		$T$ is a lattice homomorphism; or
	\item
		$p > 2$, $T$ is order bounded from $W^{1,p}(\Omega_1)$ to $L^p(\Omega_2)$, and $T(0)=0$.
	\end{enumerate}
	Then there is a pairing between the connected components of
	$\Omega_1$ and $\Omega_2$ where the respective pairs are congruent,
	i.e., $\Omega_1$ can be transformed into $\Omega_2$ by translation
	and rotation of its connected components.
	In particular, if $\Omega_1$ or $\Omega_2$ is connected, then $\Omega_1$
	and $\Omega_2$ are congruent.
\end{theorem}
\begin{proof}
	By Theorems~\ref{thm:isorig} and~\ref{thm:orderrig} the operator $T$ satisfies~\eqref{eq:Trig},
	i.e., $Tu = g \cdot (u \circ \xi)$ almost everywhere for all $u \in W^{1,p}(\Omega_1)$
	with a local rigid motion $\xi$. Since $T$ is isometric, it is injective. Hence
	$\xi$ is a bijection of $\Omega_1$ and $\Omega_2$ by Theorem~\ref{thm:W1p0eq}.
	Now the restrictions of $\xi$ to the connected components of $\Omega_2$ realize the pairing
	in the statement of the theorem.
\end{proof}

\bibliography{identification}
\bibliographystyle{amsplain}

\end{document}